\documentclass[a4paper,11pt]{article}

\usepackage{authblk}

\title{Maker-Breaker-Crossing-Game on the Triangular Grid-graph}
\author{Freddie Wallwork}
\affil{The University of Warwick}

\usepackage{amsmath}
\usepackage{amssymb}
\usepackage{amsfonts}

\usepackage{enumerate}

\usepackage[numbers]{natbib}
\usepackage{amsthm}

\usepackage{graphicx}

\theoremstyle{definition} 
\newtheorem{defn}{Definition}[section]

\theoremstyle{plain} 
\newtheorem{thm}{Theorem}[section] 
\newtheorem{lem}[thm]{Lemma}

\newtheorem*{remark}{Remark}

\begin{document}   
	\maketitle
	
	\begin{abstract}
		We study the $(p,q)$-Maker Breaker Crossing game introduced by Day and Falgas Ravry in 'Maker-Breaker percolation games I: crossing grids'. The game described in their paper involves two players Maker and Breaker who take turns claiming $p$ and $q$ as yet unclaimed edges of the graph respectively. Maker aims to make a horizontal path from a leftmost vertex to a rightmost vertex and Breaker aims to prevent this. The game is a version of the more general Shannon switching game and is played on a square grid graph. 
		
		We consider the same game played on the triangular grid graph  $\Delta_{m,n}$ ($m$ vertices across, $n$ vertices high) and aim to find, for given $(p,q,m,n)$, a winning strategy for Maker or Breaker.
		
		We establish using a similar strategy to that used by Day and Falgas Ravry to show that:
		\begin{itemize}
			\item for sufficiently tall grids and $p \geq q$ Maker has a winning strategy for the $(p,q)$-crossing game on $\Delta_{m,n}$.
			\item for sufficiently wide grids and $4p\leq q$, Breaker has a winning strategy for the $(p,q)$-crossing game on $\Delta_{m,n}$.
		\end{itemize}
		We will prove these by introducing a new type of game called a \emph{secure game} in which Maker must ensure the grid is in a specific 'secure' state at the end of her turn.
	\end{abstract}

	\section{Introduction}
	We consider the Maker-Breaker crossing game as described by Day and Falgas-Ravry on a triangular lattice grid as opposed to the square lattice grid. The results are obtained 
	through a method derived from their paper\cite{day2021makerI}.
	
	The results obtained from paper\cite{day2021makerI} are applied in Day and Falgas-Ravry's follow up paper \cite{day2021makerII} to prove results on the $(p,q)$-percolation game on the infinite square grid graph in which Breaker wins if they can completely enclose the central vertex in a dual cycle. 
	
	The proof for one of their results, that Breaker wins the $(p,q)$-percolation game for $q\geq 2p$ on the square lattice, involves building concentric square annuli which are then split into 4 crossing games for the sides and 4 corner sections. The results from \cite{day2021makerI} show that Breaker can prevent any top bottom path (or in out path) on each side of these square annuli. Another strategy for a 'box-game' means for some annulus, Breaker can claim all edges connecting the corners sections of the annulus with the side sections, hence completely enclosing the centre and winning the game. A similar result could be proven for the $(p,q)$-percolation game on the triangular or hexagonal lattice for $q>p$ or $q\geq 6p$ respectively.
	
	The results may also apply to other combinatorial games played on triangular or hexagonal grid-graphs or lattices.
	\newline
	
	\begin{defn}[$(p,q)$-crossing game]
		The Maker-Breaker $(p,q)$-crossing game played on a grid graph $G(V,E)$  is a game played by Maker and Breaker who alternately claim $p$ and $q$ as yet unclaimed edges respectively on each turn. $G(V,E)$ has a set of leftmost and rightmost vertices specified, in the case of grid graphs it is the obvious leftmost and rightmost vertices. Maker's goal is to claim edges which form a path from a leftmost vertex to a rightmost vertex. Breaker's goal is to stop this from happening.
	\end{defn}
	We will use the notation $G'(V',E')$ for the dual graph of a planar graph $G$, whose vertices are the faces of $G$ and whose edges represent the adjacency of the faces of $G$. Hence each edge $e\in G$ has a respective dual edge $e' \in G'$. 
	\newline
	
	We carry over assumptions concerning behaviour of Maker and Breaker from this paper including:
	\begin{itemize}
		\item There exists a unique winner for all finite boards
		\item For a planar graph $G(V,E)$, Breaker's goal is equivalent to claiming all edges that are dual to the edges of a top bottom path on the dual graph $G'$ of G
		\item Maker never makes a cycle or arch (from left to left or right to right)
		\item Breaker never makes a dual-cycle or dual-arch (from top to top or bottom to bottom)
		\item Claiming an edge is never disadvantageous to either player 
	\end{itemize} 

	Something we cannot carry over from Day and Falgas-Ravry's paper is the grid being self dual as is the case with the square grid lattice $\Lambda_{n+1,n}$. The dual of the triangular grid graph $\Delta_{m,n}$ is in fact the hexagonal grid graph ($H_{n,m}$) which we will address in section 4. This makes determining an exact threshold for c in the $(p,cp)$-crossing game challenging. We know from our two main results that the value for c is somewhere between 1 and 4
	\begin{defn}
		We denote $\Delta_{m,n}$ as the graph formed by a triangular lattice with n rows with $m$ and $m-1$ vertices per row alternating upward. \\
		We number the rows upward from $1$ to $n$ and vertices are horizontally labelled similarly with 2 horizontally adjacent vertices labelled a difference of $2$ from each other.\\
	\end{defn}
	
\begin{figure}[h]
	\centering
	\includegraphics[scale=0.5]{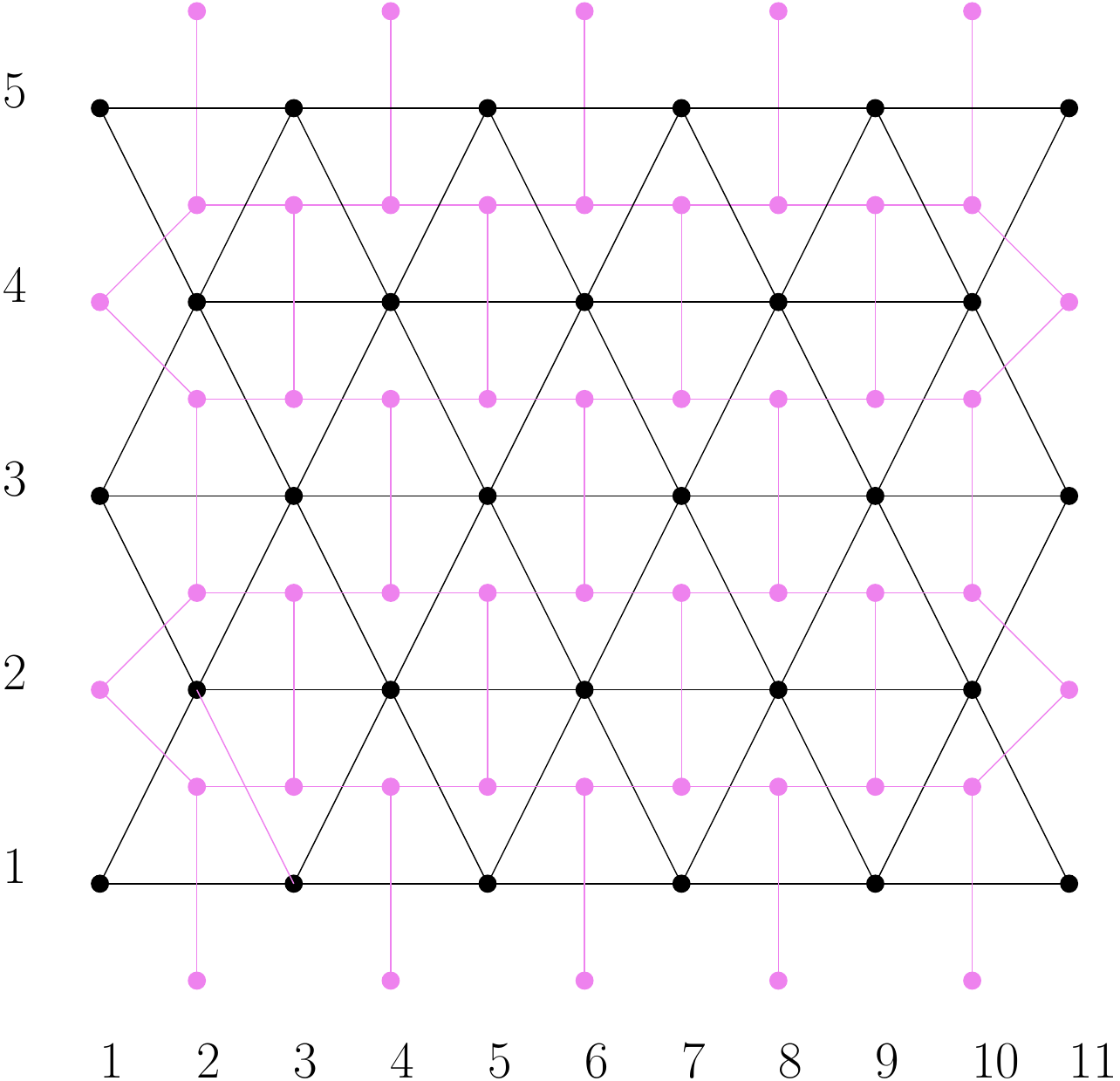}
	\caption{The grid $\Delta_{6,5}$ in black and it's 'dual' in violet}
\end{figure}
	
We can observe that the \emph{dual} of this graph $\Delta'_{m,n}$ is the hexagonal lattice and Breakers objective of stopping Maker creating a left right crossing path is equivalent to creating a top bottom dual crossing path on this dual graph.

We will label the dual-vertices intuitively horizontally and give all vertices between standard vertex row $k$ and $k+1$ the vertical label $k+\frac{1}{2}$. (We will later define a 'dual vertex level' to distinguish dual-vertices representing the faces of upward and downward orientated triangles with the same vertical label.)

\begin{defn}[External boundary edges]
	We define the \emph{external boundary edges} of a connected component $C$ of Maker edges to be the unclaimed edges from the set of vertices in component $C$, $V(C)$, to the set of vertices not in $C$, $V(G) \backslash V(C)$. We denote this $\beta(C)$.
\end{defn}

\begin{lem}
	The number of edges in the external boundary of a connected component $C$ on the dual graph $\Delta'{m,n}$  is at most 3 more than the edges in $C$ .
	\[ |\beta(C)| \leq |E(C)| + 3 \] 
\end{lem}
\begin{proof}
	Proof is proven easily by induction
	
	A single dual edge has at most 4 external boundary edges.\\
	Consider a dual component of $n$ edges. If there is a cycle in the dual component then there exists a dual edge $e$ such that $C\backslash e$ is a connected component of $n-1$ dual edges. Adding $e$ back in makes no difference to the set of external boundary edges. So by inductive hypothesis $|\beta(C)| \leq (n-1) + 3 \leq n+3$.
	
	If there is no cycle in $C$ then the component is a tree. Let $e$ be adjacent to a leaf of C. $C \backslash e$ is a tree of $n-1$ edges. By inductive hypothesis  $|\beta(C\backslash e)| \leq (n-1) + 3$. Adding $e$ adds at most 2 dual edges to the set of external boundary edges but removes $e$. Hence, $|\beta(C)| \leq n+3$.
\end{proof}

We see from this lemma that with an allowance of 3 edges per component a component of Breaker edges could be 'contained' by Maker. Hence given a sufficient strategy to manage these 3 edges a winning strategy for Maker would follow for the $(p,q)$-crossing game when $p \geq q$.
	
The following is the first main result we aim to prove.
	
\begin{thm}
	Maker has a winning strategy for the $(p,q)$-crossing-game on the triangular grid graph $\Delta_{m,n}$ for $p\geq q$ and $n\geq q+2$
\end{thm}
Since it is never disadvantageous for either player to claim a free edge, we only need to prove this for the case $p=q$.

\section{The $q$-response game, Brackets and Security}
We define a more general variant of the $(p,p)$-crossing game between Maker and Breaker.
\begin{defn}{$q$-response game}
	The $q$-response game is played on $\Delta_{\infty,n}$ by players Horizontal ($\mathcal{H}$) and Vertical ($\mathcal{V}$). 
	On each turn $t$, $\mathcal{V}$ picks an integer $r_t \in [q]$ and claims $r_t$ unclaimed edges in $\Delta_{\infty,n}$ marking them red. $\mathcal{H}$ responds to each of $\mathcal{V}$ turns by also claiming $r_t$ unclaimed edges, marking them blue. 
	The goals for Vertical ($\mathcal{V}$) is equivalent to that of Breaker. Horizontal ($\mathcal{H}$) wins if she can prevent $\mathcal{V}$ from getting a top-bottom path.
\end{defn}

The (q,q)-crossing game is the special case of the response game where $\mathcal{V}$ chooses $r_t=q$ for all $t$ and $\mathcal{H}$ forfeits her first turn. Therefore note, any winning strategy for $\mathcal{H}$ on the $q$-response game is also a winning strategy for Maker in the $(p,p)$-crossing game where Maker may make an arbitrary first move.
	
\begin{defn}
	We refer to a maximal set of $2$ or more dual-vertices pairwise connected by a path of red edges as a:
	\begin{itemize}
		\item \emph{top component} if it has at least 1 vertex from the top of the grid
		\item \emph{bottom component} if it has at least 1 vertex from the bottom of the grid
		\item \emph{floating component} otherwise			
	\end{itemize}
\end{defn}

\begin{defn}[Brackets]
	We define Brackets as $3$ edge long paths from vertex $(x,y)$ in the following ways:
	
	\begin{description}
		\item[Type 1] the edges $\{ (x-0.5,y-0.5) , (x,y-1) , (x+1.5, y-0.5)\}$ form a bracket of Type 1 if none are red.\\
		The interior dual-vertices of this bracket are defined as $\{ (x,y-0.5), (x+1,y-0.5)\}$.
		
		\item[Type 2] the edges $\{ (x+0.5,y-0.5) , (x+2,y-1) , (x+2.5, y-0.5)\}$ form a bracket of Type 2 if none are red.\\
		The interior dual-vertices of this bracket are defined as $\{ (x+1,y-0.5), (x+2,y-0.5)\}$.
		
		\item[Type 3] the edges $\{ (x+0.5,y-0.5) , (x+1.5,y-0.5) , (x+1.5, y+0.5)\}$ form a bracket of Type 3 if none are red.
		The interior dual-vertices of this bracket are defined as $\{ (x+1,y-0.5), (x+1,y+0.5)\}$.
		\begin{figure}[h]
			\centering
			\includegraphics[scale=0.6]{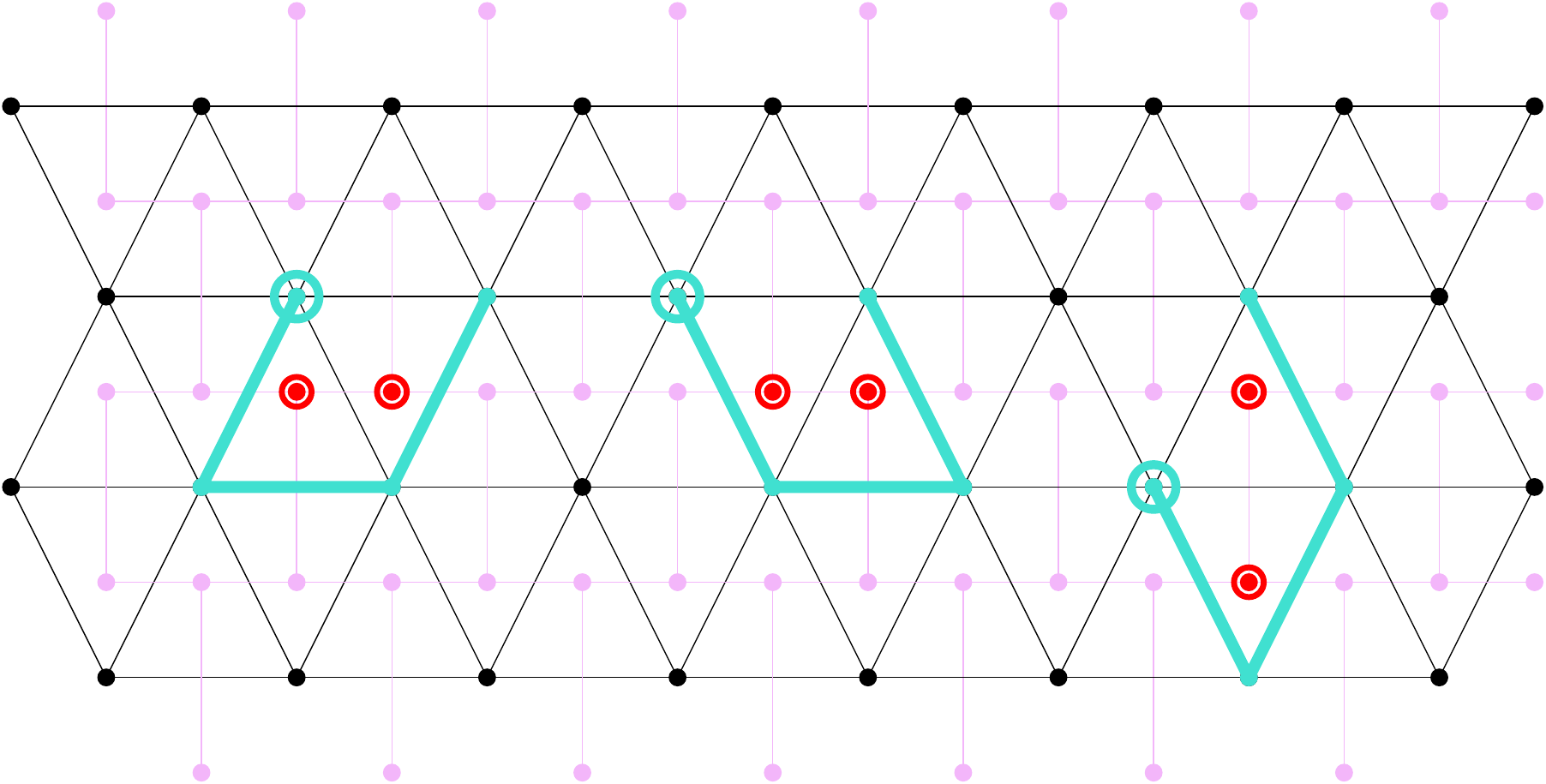}
			\caption{Bracket Types 1, 2 \& 3 with $(x,y)$ marked with turquoise circle and interior dual-vertices in red}
		\end{figure}
			
		\item[Type 4] the edges $\{ (x+0.5,y-0.5) , (x+2,y-1) , (x+3.5, y-0.5)\}$ form a bracket of Type 4 if none are red.
		The interior dual-vertices of this bracket are defined as $\{ (x+1,y-0.5), (x+2,y-0.5) , (x+3,y-0.5)\}$.
			
		\item[Type 5] the edges $\{ (x+1,y) , (x+3,y) , (x+3.5, y+0.5)\}$ form a bracket of Type 5 if none are red.
		The interior dual-vertices of this bracket are defined as $\{ (x+1,y+0.5), (x+2,y+0.5) , (x+3,0.5)\}$.
		
		\item[Type 6] the edges $\{ (x+1,y) , (x+3,y) , (x+4.5, y+0.5)\}$ form a bracket of Type 6 if none are red.
		The interior dual-vertices of this bracket are defined as $\{ (x+1,y+0.5), (x+3,y+0.5) , (x+4,y+0.5)\}$.
		
		\begin{figure}[h]
			\includegraphics[scale=0.6]{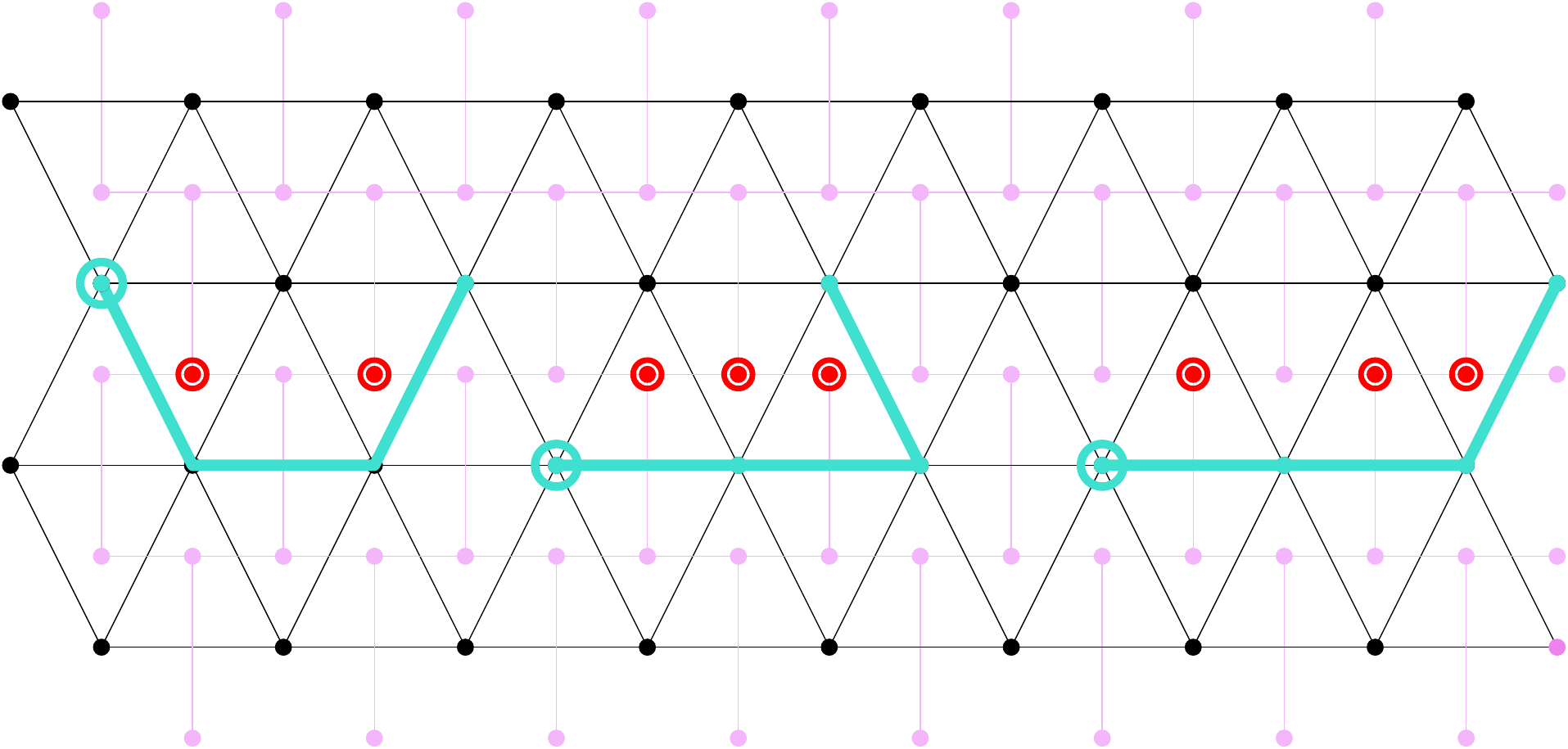}
			\caption{Bracket Types 4, 5 \& 6 with $(x,y)$ marked with turquoise circle and interior dual-vertices in red}
		\end{figure}
			
		\item[Type 7] the edges $\{ (x+0.5,y-0.5) , (x+1.5,y-0.5) , (x+2.5, y+0.5)\}$ form a bracket of Type 7 if none are red.
		The interior dual-vertices of this bracket are defined as $\{ (x+1,y-0.5), (x+1,y+0.5), (x+2, Y+0.5)\}$.
		
		\item[Type 8] the edges $\{ (x+0.5,y-0.5) , (x+1.5,y-0.5) , (x+3, y)\}$ form a bracket of Type 8 if none are red.
		The interior dual-vertices of this bracket are defined as $\{ (x+1,y-0.5), (x+1,y+0.5), (x+2, y+0.5)\}$.
		\begin{figure}[h]
			\includegraphics[scale=0.6]{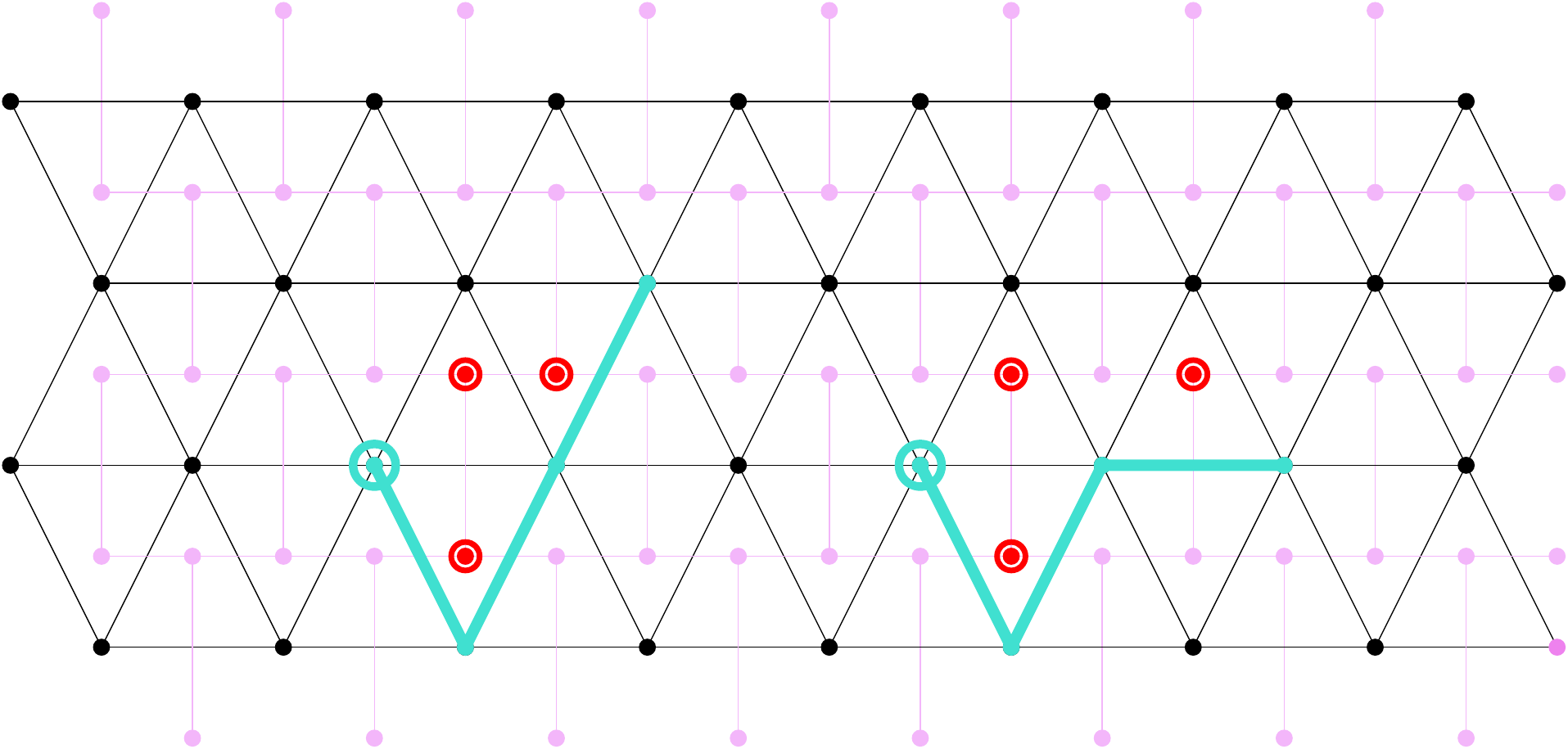}
			\caption{Bracket Types 7 \& 8 with $(x,y)$ marked with turquoise circle and interior dual-vertices in red}
		\end{figure}
	\end{description}
\end{defn}

\begin{defn}[Secure floating component]
	We say a floating component $C$ is a secure floating component secured by a blue path $P$ and bracket $B$ if:
	\begin{enumerate}[(i)]
		\item The endpoints of $B$ and $P$ are the same
		\item The interior dual-vertices of the bracket $B$ is the vertex set of $C$
		\item $C$ is in the the interior of the cycle formed by $P\cup B$
		\item The dual $e*$ of each edge $e$ of $P$ is adjacent to a vertex of $C$ 
	\end{enumerate}
\end{defn}
	
If $C$ is a top (/bottom) component it has a unique top (/bottom) dual-vertex $v*$ (since $\mathcal{V}$ never makes a dual-arch).
\begin{defn}[Secure top component]
	We say a top component $C$ with top dual-vertex $(x,n+0.5)$ is a secure top component secured by a non-red edge $g=(x',n-0.5)$ with $x'>x$ and a blue path $P$ from $(x-1,n)$ to the lower endpoint of $g$  if:
	\begin{enumerate}[(i)]
		\item $C$ is in the the interior of the arch $P\cup \{g\}$
		\item The dual $e*$ of each edge $e$ of $P$ is adjacent to a vertex of $C$ 
	\end{enumerate}
\end{defn}

\begin{figure}[h]
	\includegraphics[scale=0.6]{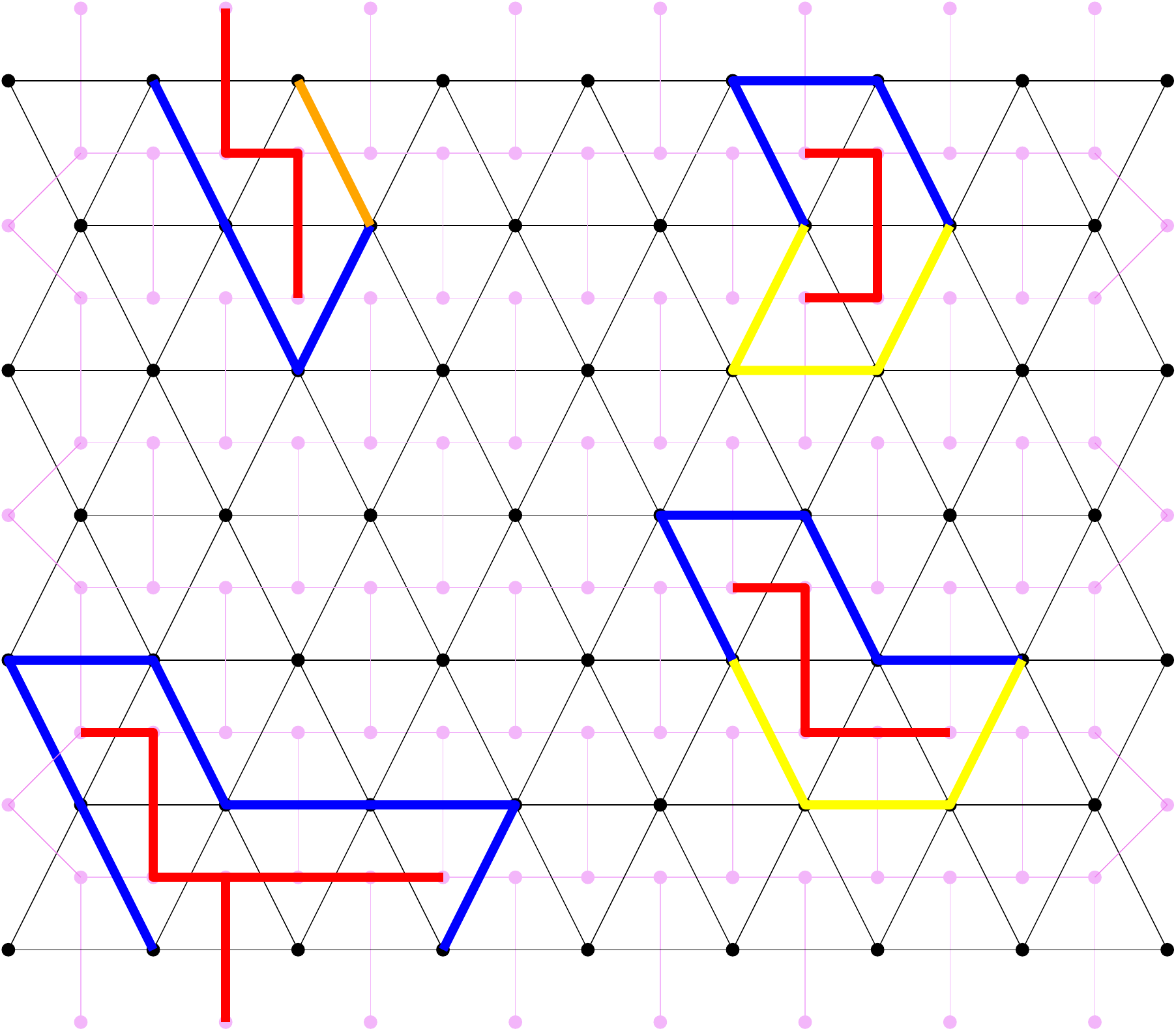}
	\caption{from top left clockwise: A secure top component secured by blue path and orange 'gate', a secure floating component secured by blue path and Type 1 bracket, a secure floating component secured by blue path and Type 3 bracket and an extra secure bottom component secured by blue path}
\end{figure}

\begin{defn}[Secure bottom component]
	We say a bottom component $C$ with bottom dual-vertex $(x,-0.5)$ is a secure top component secured by a non-red edge $g=(x',1.5)$ with $x'>x$ and a blue path $P$ from $(x-1,1)$ to the upper endpoint of $g$  if:
	\begin{enumerate}[(i)]
		\item $C$ is in the the interior of the arch $P\cup \{g\}$
		\item The dual $e*$ of each edge $e$ of $P$ is adjacent to a vertex of $C$
	\end{enumerate}
\end{defn}

We refer to the edge $g$ in the definition of Secure top/bottom component as the 'gate' and $g-(0.5,0)$ the interior dual-vertex. If $g$ is also a blue edge then the component is 'extra secure'.\\

For the dual graph $\Delta'_{\infty,n}$ we label the vertical 'dual vertex level' of $v \in V(\Delta'_{m,n})$ intuitively from bottom to top: $1$ to $2n$. With the dual-vertex at the centre of each upward pointing triangle with base on the $k^{th}$ line being on dual vertex level $2k$ and the dual-vertex at the centre of a downward pointing triangle with point on the $k^{th}$ line being on dual vertex level $2k+1$. See

\begin{figure}[h]
	\includegraphics[scale=0.6]{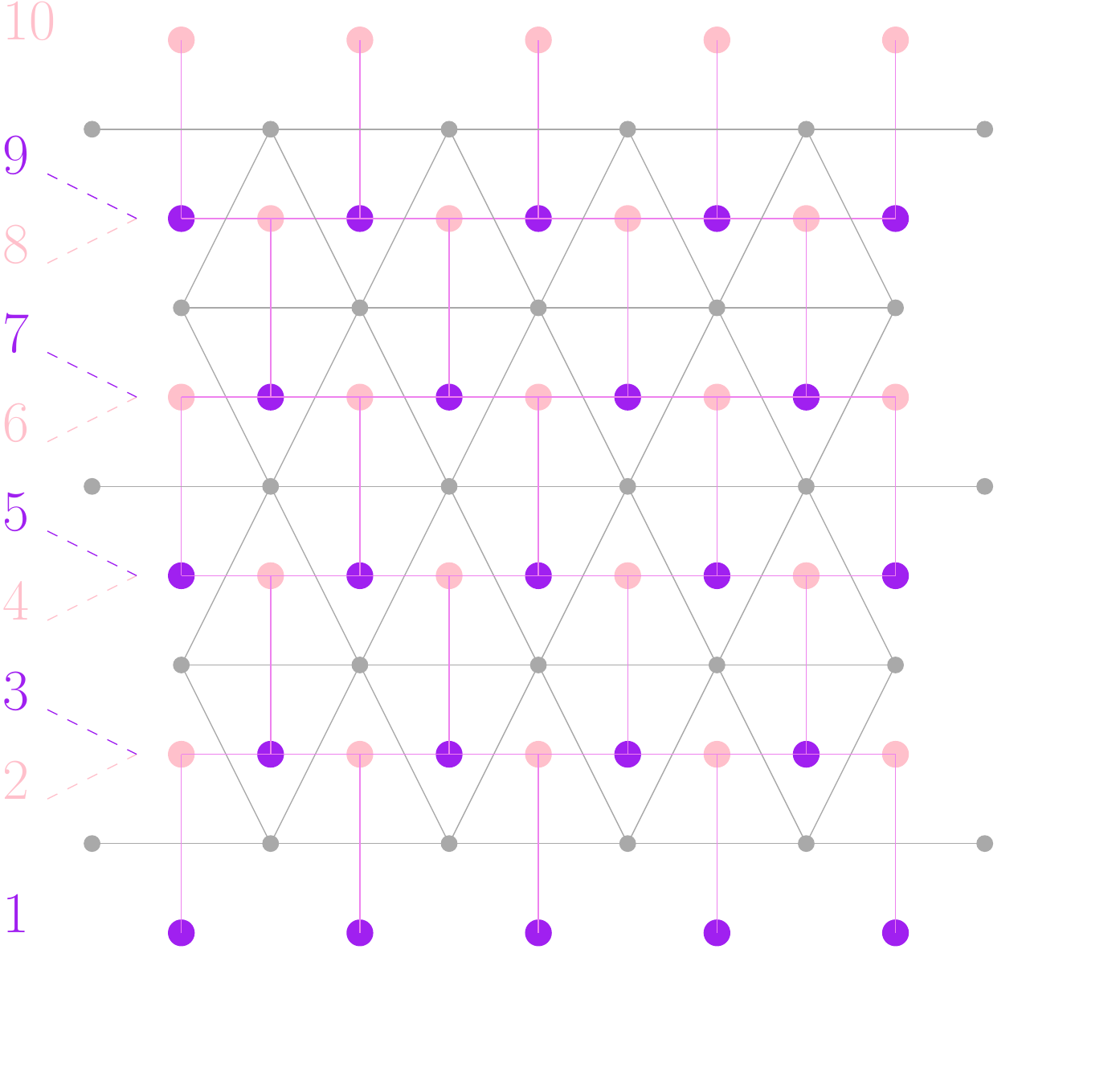}
	\caption{a triangular grid graph with annotated dual vertex levels}
\end{figure}

\begin{lem}
	let $n\geq q+2$. If the grid is secure at the start of $\mathcal{V}$'s turn in the $q$-response game played on $\Delta_{\infty,n}$, then $\mathcal{V}$ cannot win in a single turn.
\end{lem}
\begin{proof}
	Let the grid be in a secure position before $\mathcal{V}$'s turn. Suppose he then claim $l$ edges $\{e_1,e_2,\dots,e_l\}$ to create a bottom top crossing dual path of red edges.
	
	Let the order of the edges be that of their appearance in the path from bottom to top. 
	We define $v^{-}_i$ and $v^{+}_i$ to be the 'dual vertex level' of the first and second vertex adjacent to $e_i$ in order of appearance in the bottom top path.
	
	We show by induction on $k$ that $v_k^{+}\leq 2k+1$\\
	For $k=1$ either $e_1$ is adjacent to a bottom dual-vertex (in which case $v^{+}_1=2$) or it meets a bottom component. As the bottom components is secure $e_1$ must lie across the gate. Hence $e_1=(x',1.5)$ and consequently  $v^{+}_1\leq 3$. Hence the statement holds for $k=1$
	
	Now let us suppose $v^{+}_k \leq 2k+1$ and consider the dual edge $e_{k+1}$.
	If $e_k$ and $e_{k+1}$ share a vertex, then since any 2 adjacent vertices are at most 1 dual vertex level apart, $v_{k+1}^{+}\leq v_k^{+}+1 \leq 2(k+1)$.
	If $e_k$ and $e_{k+1}$ do not share a vertex then there must be a secure floating component $C$ between them. If $C$ is secured by a bracket $B$, then consider the dual of each edge in $B$. These are the possible edges of  $e_k$ and $e_{k+1}$. We consider the greatest dual vertex level reached by any of these dual edges with respect to the lowest interior dual-vertex as this is the highest possible value for $v_{k+1}^{+}- v_k^{+}$.
	
	By observation, we can see that bracket type 3 has largest difference of 2 vertex levels. Therefore,
	\begin{align*}
		v_{k+1}^{+}- v_k^{+} \leq 2\\
		v_{k+1}^{+} \leq v_k^{+}+2\\
		v_{k+1}^{+} \leq 2k+3
	\end{align*}
	Hence, $v_{l}^{+}\leq 2l+1 \leq 2n-3$. Since the interior dual-vertex of the gate of a secure top component is at dual vertex level of at least $2n-2$, $e_l$ does not reach a top vertex nor top component. 
	
	Therefore $\mathcal{V}$ cannot win in a single turn from a secure grid.
\end{proof}
	
\section{The secure game}
	
\begin{defn}{Secure game}
	The Secure game is played by $\mathcal{H}$ and $\mathcal{V}$ on the graph $\Delta_{\infty,n}$. At any point an edge may be unclaimed, red, blue or a blue double edge (claimed by $\mathcal{H}$ twice).
	
	On each turn $\mathcal{V}$ claims ANY edge and makes it red, subject to the following conditions:
	\begin{enumerate}[(i)]
		\item $\mathcal{V}$ cannot claim an edge if doing so would create either a red cycle or a red arch (path from top dual-vertex to top dual-vertex or equivalently for bottom dual-vertice).
		\item $\mathcal{V}$ cannot claim an edge to connect a top component and bottom component
		\item If $C$ is a floating component and $P$ a blue path that 'secures' $C$, $\mathcal{V}$ cannot claim an edge of $P$ if doing so turns $C$ into a top or bottom component.
	\end{enumerate}
	Once $\mathcal{V}$ has claimed an edge $e$, $\mathcal{H}$ responds by claiming $1+b$ unclaimed edges, where $b$ is the number of blue edges replaced by $e$.
	
	At any stage the grid is secure if:
	\begin{enumerate}
		\item the grid is 'secure' with respect to the q-response game
		\item for distinct red components $C$ and $C'$ secured by blue paths $P$ and $P'$, any edge in $P \cap P'$ is a blue double edge.
	\end{enumerate}
	$\mathcal{H}$ wins if she can ensure the grid is secure at the end of each of their goes. $\mathcal{V}$ wins otherwise.
\end{defn}

\begin{lem}
	$\mathcal{H}$ has a winning strategy for for the secure game on $\Delta_{\infty,n}$ for $n\geq2$.
\end{lem}
\begin{proof}
	Suppose the grid is secure before the beginning of $\mathcal{V}$'s turn and let $\mathcal{V}$ claim edge $e$ where $e*$ is adjacent to dual-vertices $v_1$ and $v_2$.\\
	
	\textbf{Case 1.} Let $v_1$ and $v_2$ not be part of any existing red components.\\
	If neither $v_1$ or $v_2$ are top/bottom vertices then clearly by claiming the edge above  or top left of $e$, $\mathcal{H}$ can secure the new component with this edge and a bracket of type 1,2 or 3 depending on orientation of $e$. Else if WLOG $v_1$ is a top (/bottom) vertex then by claiming the edge bottom (/top) left of $v_2$ it is a secure top (/bottom) component.
	
	\textbf{Case 2.} Let $v_1$ be part of an existing red component $C$ and $v_2$ be not. Let $C$ be secured by blue path $P$ and bracket of non-red edges $B$.
	
	\emph{Case 2a}. If $e$ lies in the interior then $C$ is still secure.
	
	\emph{Case 2b}. If $e$ is a blue edge then by the conditions of the game $v_2$ cannot be a bottom or top vertex and $\mathcal{H}$ can claim two edges. She claim the non-red edges $f_1$ and $f_2$ whose dual are adjacent to $v_2$ to 're-secure' the component with the the new blue path $P-{e}+{f_1,f_2}$ and the same bracket $B$.
	
	\emph{Case 2c}. If $e$ is the 'gate' of a top (or bottom) component $C$ then we consider the vertex level of $v_2$.
	
	If $v_2$ is on dual vertex level $2n-1$ (or $2$ for bottom) and $e$ goes back on the component, then player $\mathcal{H}$ claims edge with midpoint $(v_2)+(0.5,0)$ to make the component extra secure.
	
	If $v_2$ is on dual vertex level $2n-2$ (or $3$ for bottom component) then player $\mathcal{H}$ claims edge with midpoint $(v_2)+(0,-0.5)$ (or $(v_2)+(0,0.5)$ for bottom component). This makes the new component secure by the definition.\\
		
	\emph{Case 2d}. If $e$ is part of the bracket $B$ securing the existing red component then we consider cases on each edge of each bracket.
	
	The cases to check, for example we will consider bracket type 1 with end vertices $(x,y)$ and $(x+2,y)$:
	\begin{itemize}
		\item If $e=(x-0.5,y-0.5)$ then $\mathcal{H}$ claims edge $f=(x-1,y)$ to secure the component with $P\cup f$ and bracket of type 4 with end vertices $(x-2,y)$ and $(x+2,y)$.
		\item If $e=(x,y-1)$ then $\mathcal{H}$ claims edge $f=(x-0.5,y-0.5)$ to secure the component with $P\cup f$ and bracket of type 7 with end vertices $(x-1,y-1)$ and $(x+2,y)$.
		\item If $e=(x+1.5,y-0.5)$ then $\mathcal{H}$ claims edge $f=(x-0.5,y-0.5)$ to secure the component with $P\cup f$ and bracket of type 5 with end vertices $(x-1,y-1)$ and $(x+2,y)$.
	\end{itemize}
	The cases of the other bracket types are reacted to analogously by $\mathcal{H}$.
		
	\textbf{Case 3.} Let $v_1$ and $v_2$ be part of existing components $C_1$ and $C_2$ . Let $C_i$ be secured by blue path $P_i$ and either bracket of non-red edges $B_i$ or gate of non-red edge $g_i$ if $C_i$ is not secure. \\
	Clearly $C_1$ cannot equal $C_2$ else $e*$ creates a dual cycle in that connected component which breaks rule (i).\\
	$C_1$ and $C_2$ cannot both be top or bottom components else $e$ either creates a dual arch or connects a top and bottom component contradicting conditions (i) and (ii) respectively of the secure game.\\
	
	Case 3a. $C_1$ is a top or bottom component and $C_2$ is a floating component. By restriction (iii) $e$ must be in the bracket $B_2$. If $e$ is part of the blue path $P_1$ then player $\mathcal{H}$ has 2 moves and chooses the other 2 edges of $B_2$ to make $C_1 \cup C_2 \cup {e*}$ a secure top or bottom component. If $e$ is the gate $g_1$ then we consider cases of each bracket.
	
	There is no positioning of brackets type 5 and 6 that align with gates of top or bottom components. 
	
	Brackets of type 3, 4, 7 and 8 are only compatible with gates of bottom components. For all instances the left most edge of the bracket is equal to the gate edge.  $\mathcal{H}$ claims the right most edge and leaves the middle edge of the brackets to be the new gate of the new bottom component.
	Brackets of type 1, 2 and 4 are compatible with gates for both top and bottom components. Again, for all instances the left most edge is equal to the gate edge. $\mathcal{H}$ claims the middle edge of the bracket leaving the right edge to be the new gate of the new top or bottom component.
	
	Case 3b. $C_1$ and $C_2$ are both floating components. If $e$ is part of both $P_1$ and $P_2$, then player $\mathcal{H}$ has 3 moves and uses them to claim all edges of bracket $B_1$. If $e$ is part of $P_1$ and $B_2$, then player $\mathcal{H}$ has 2 moves and uses them to claim the remaining 2 edges of bracket $B_2$.\\
	Finally if $e$ is part of $B_1$ and $B_2$, then we consider cases on brackets.\\
	Cases are easy to check since many brackets aren't compatible with one another. We will check The possible instances for the bracket of type 1:
	\begin{itemize}
		\item Alignment is possible with another bracket of type 1 where the third and first edge is shared respectively. $\mathcal{H}$ can claim the first edge of the left bracket $f$ to secure the new component with $P_1\cup P_2\cup f$ and bracket of type 6.
		\item Alignment is possible between brackets of type 1 and 3, sharing first and second edges respectively. $\mathcal{H}$ claims the third edge of the type 3 bracket $f$ to secure the new component with $P_1\cup P_2\cup f$ and bracket of type 4.
		\item Alignment is possible between brackets of type 1 and 4, sharing first and third edges respectively. $\mathcal{H}$ claims the first edge of the type 4 bracket $f$ to secure the new component with $P_1\cup P_2\cup f$ and bracket of type 6.
		\item Alignment is possible between brackets of type 1 and 6, sharing first and third edges respectively. $\mathcal{H}$ claims the first edge of the type 6 bracket $f$ to secure the new component with $P_1\cup P_2\cup f$ and bracket of type 6.
		\item Alignment is possible between brackets of type 1 and 7, sharing first and second (/or third) edges respectively. $\mathcal{H}$ claims the third edge of the type 7 bracket (/or third of the type 1 bracket) $f$ to secure the new component with $P_1\cup P_2\cup f$ and bracket of type 4 (/or 8).
		\item Alignment is possible between brackets of type 1 and 8, sharing first and second edges respectively. As the third edge of bracket type 8 is now not an external boundary edge the new component is secured by $P_1\cup P_2$ and a bracket of type 4.
		\item No alignment is possible between brackets of type 1 and type 2 and 5. 
	\end{itemize}
	
	Hence, at the end of player $\mathcal{H}$'s turn, the board is always secure. Hence $\mathcal{H}$ has a winning strategy for the secure game on $\Delta_{\infty,n}$ for $n\geq2$.
		
\end{proof}
	
Now, we apply the strategy for the secure game to the q-response game to get our final result.

\begin{thm}
	Maker has a winning strategy for the q-response secure game on $\Delta_{\infty,n}$ for $n\geq q+2$.
\end{thm}
\begin{proof}
	We aim to prove that player $\mathcal{H}$  can keep the board secure at the end of their go.\\
	Suppose the grid is secure at the start of $\mathcal{V}$'s turn and he claim $r\leq q$ edges $A={e_1, e_2, \dots e_r}$ in their turn. By Lemma 2.1 $\mathcal{V}$ cannot win in this single turn.

	Player $\mathcal{H}$ responds to each of these edges one by one as if she were played in the secure game. If an edge $e_j$ was claimed by $\mathcal{H}$ is response to a previous edge $e_i$ then $\mathcal{H}$ may use this edge elsewhere as in the secure game. After responding to all edges in  $A$, player $\mathcal{H}$ has a set $B$ of at most $r$ unselected edges. The grid remains secure.
	
	However, we must check what happens if player $\mathcal{V}$ contradicts conditions (i)-(iii) of the secure game:
	\begin{enumerate}[(i)]
		\item We can assume player $\mathcal{V}$ never creates a dual cycle or dual arch in the q-response game. So condition (i) is satisfied.
		\item For the second condition $\mathcal{V}$ cannot claim a dual edge if doing so connects a top and bottom component. This would mean that $\mathcal{V}$ wins and, by Lemma 2.1, $\mathcal{V}$ cannot win in one turn from a secure grid. So condition (ii) is satisfied.
		\item Finally, there is condition (iii), which states If $C$ is a floating component and $P$ a blue path that 'secures' $C$, $\mathcal{V}$ cannot claim an edge of $P$ if doing so turns $C$ into a top or bottom component.\\
		For this we choose an order for $A$ for $\mathcal{H}$ to respond. Due to conditions (i) and (ii), all top and bottom components are rooted trees with the root being the unique top or bottom dual-vertex. Let $A$ be ordered such that:\\
		If $e,e* \in A$ are dual edges of the same top or bottom component after $\mathcal{V}$'s turn, and $e$ is strictly closer in graphical distance to the root, then $e$ appears before $e*$ in $A$.\\
		Suppose, using this ordering, that there exists an edge $e_i \in A$ such that $\mathcal{V}$ claiming $e_i$ violates condition (iii). Hence, before $e_i$ is played a floating component $C$ is secured by blue path $P$ and bracket $B$ is made a top or bottom vertex after $e_i$ is played. Given the ordering of $A$, no other edge of $A$ adjacent to $C$ could have been played before $e_i$ in the secure game as $e_i$ is what joins $C$ to the bottom or top component and therefore has a strictly smaller graphical distance to the root. As none of the dual edges adjacent to $C$ were played before $e_i$ all edges of $P$ were present before $\mathcal{V}$ started their turn.\\ Therefore in the q-response game $\mathcal{V}$ could not claim $e_i$ hence a contradiction.\\
		As player $\mathcal{H}$ can hence ensure the grid is always secure at the start of $\mathcal{V}$'s turn and he can't win in one turn from a secure grid, then as the grid is finite, eventually $\mathcal{H}$ will achieve a left right crossing path. This is a winning strategy for Horizontal player
	\end{enumerate}		
\end{proof}

\section{A similar strategy for Breaker}
We wish to prove that Breaker has a winning strategy in the $(p,4p)$-Crossing game.

\begin{thm}
	Breaker has a winning strategy for the $(p,q)$-crossing-game on the triangular grid graph $\Delta_{m,n}$ for $q\geq 4p$ and $m\geq q+1$
\end{thm}

Using the fact that the Breaker's objective on a planar grid graph $G$ is equivalent to Maker's objective on a the dual graph $G'$ rotated $90^\circ$, we can replicate the method used to get a strategy for Maker. 
	
It is important to note, however,  that as the Maker of the original game has first move, that Breaker is the first to  in this 'dual game'. To differentiate these games Maker and Breaker on this dual grid of this game will be referred to as 'DualMaker' and 'DualBreaker'.
	
We will call the dual graph generated as above by $\Delta_{m,n}$, $H_{n,m}$ since the dual of the triangular grid is hexagonal. On this new grid graph, DualMaker plays on the Hexagonal grid graph and DualBreaker plays on the triangular grid graph.
	
\begin{figure}[h]
	\includegraphics[scale=0.6]{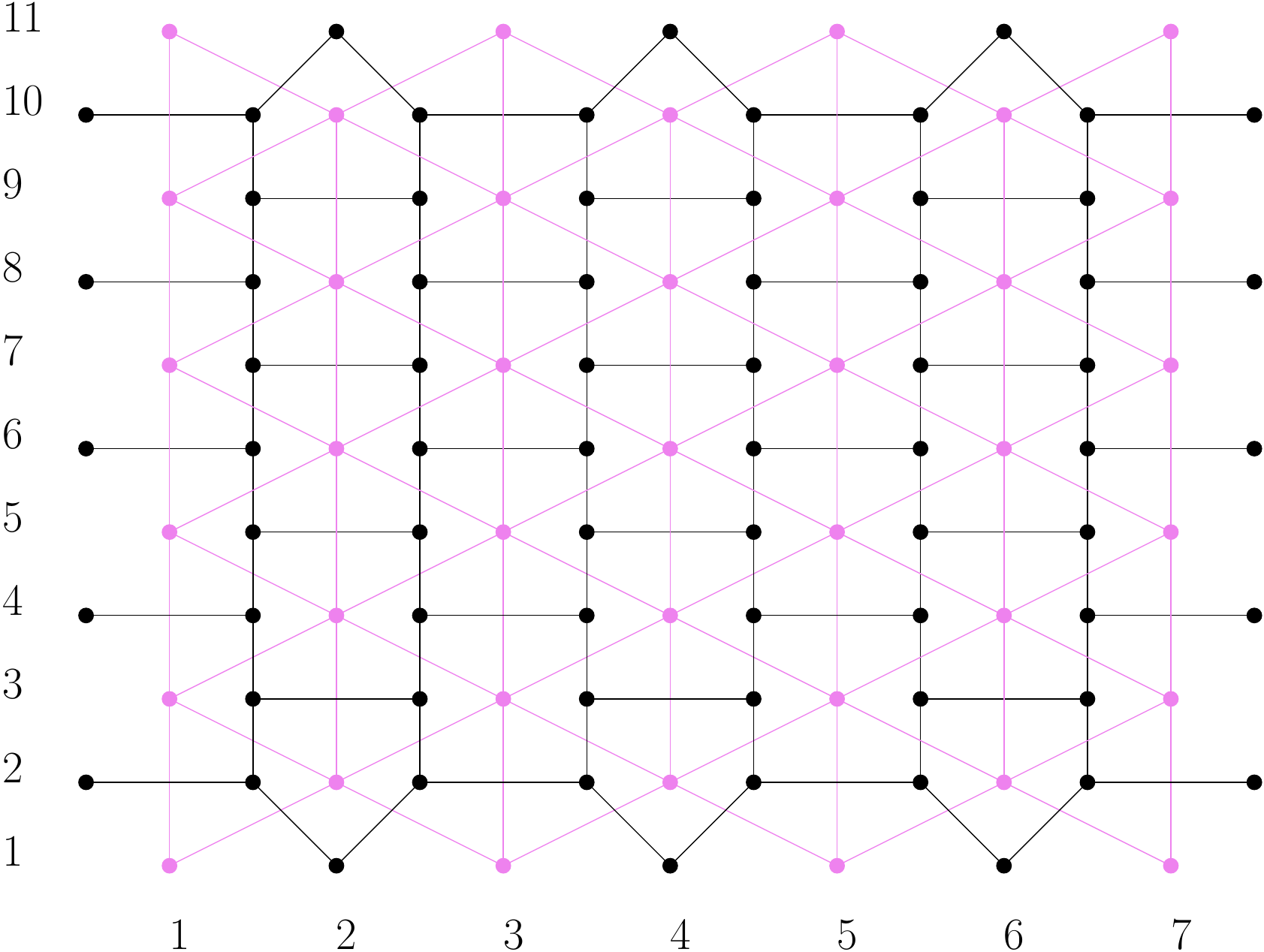}
	\caption{The grid graph $H_{7,6}$}
\end{figure}
	
Again we consider a response game though this time the Horizontal player may respond with 4 times the moves of Vertical.
	
\begin{defn}{q-4response game}
	The q-4response game is played on $H_{\infty,n}$ by players Horizontal ($\mathcal{H}$) and Vertical ($\mathcal{V}$). 
	On each turn $t$, $\mathcal{V}$ picks an integer $r_t \in [q]$ and claims $r_t$ unclaimed edges in $H_{\infty,m}$ marking them red. $\mathcal{H}$ responds to each of $\mathcal{V}$'s turns by also claiming $4r_t$ unclaimed edges, marking them blue. 
	The goals for Vertical ($\mathcal{V}$) is equivalent to that of Breaker. Horizontal ($\mathcal{H}$) wins if she can prevent $\mathcal{V}$ from getting a top bottom path.
\end{defn}

Again, the crossing-game is a specific version of the q-4response game where $r_t=q \;\;\forall t$ , a winning strategy for horizontal on this game can be used directly by DualMaker as a winning strategy for the $(p,4p)$-crossing game on $H_{m,n}$ for any $m\geq1$.

\begin{defn}[DualBrackets]
	We define DualBrackets similarly to Brackets, as 6 edge-long paths from vertex $(x,y)$, in the following ways:
		
	\begin{description}
		\item[Type 1] the edges $\{ (x,y-0.5) , (x+0.5,y-1) , (x+1, y-0.5) , (x+1, y+0.5) , (x+1, y+1.5) , (x+1, y+2.5)\}$ form a bracket of Type 1 if none are red.\\
		The interior dual-vertices of this bracket are defined as $\{ (x+0.5,y), (x+0.5,y+2)\}$.
		
		\item[Type 2] the edges $\{ (x,y-0.5) , (x+0.5,y-1) , (x+1, y-0.5) , (x+1.5,y) , (x+2,y+0.5) , (x+2,y+1.5)\}$ form a bracket of Type 2 if none are red.\\
		The interior dual-vertices of this bracket are defined as $\{ (x+0.5,y), (x+1.5,y+1)\}$.
		
		\item[Type 3] the edges $\{ (x,y-0.5) , (x+0.5,y-1) , (x+1, y-0.5) , (x+1,y+0.5) , (x+0.5,y+1) , (x,y+1.5)\}$ form a bracket of Type 3 if none are red.
		The interior dual-vertices of this bracket are defined as $\{ (x,y+0.5), (x-0.5,y+0.5)\}$.
			
		\item[Type 4] the edges $\{ (x,y-0.5) , (x+0.5,y-1) , (x+1, y-0.5) , (x+1, y+0.5) , (x+1,y+1.5) , (x+1.5,y+2)\}$ form a bracket of Type 4 if none are red.
		The interior dual-vertices of this bracket are defined as $\{ (x+0.5,y), (x+0.5,y+2) , (x+1.5,y+3)\}$.
		
		\item[Type 5] the edges $\{ (x,y-0.5) , (x+0.5,y-1) , (x+1, y-0.5) , (x+1.5,y) , (x+2,y+0.5) , (x+2.5,y+1)\}$ form a bracket of Type 5 if none are red.\\
		The interior dual-vertices of this bracket are defined as $\{ (x+0.5,y), (x+1.5,y+1), (x+2.5,y+2)\}$.
		
	\begin{figure}[h]
		\centering
		\includegraphics[scale=0.6]{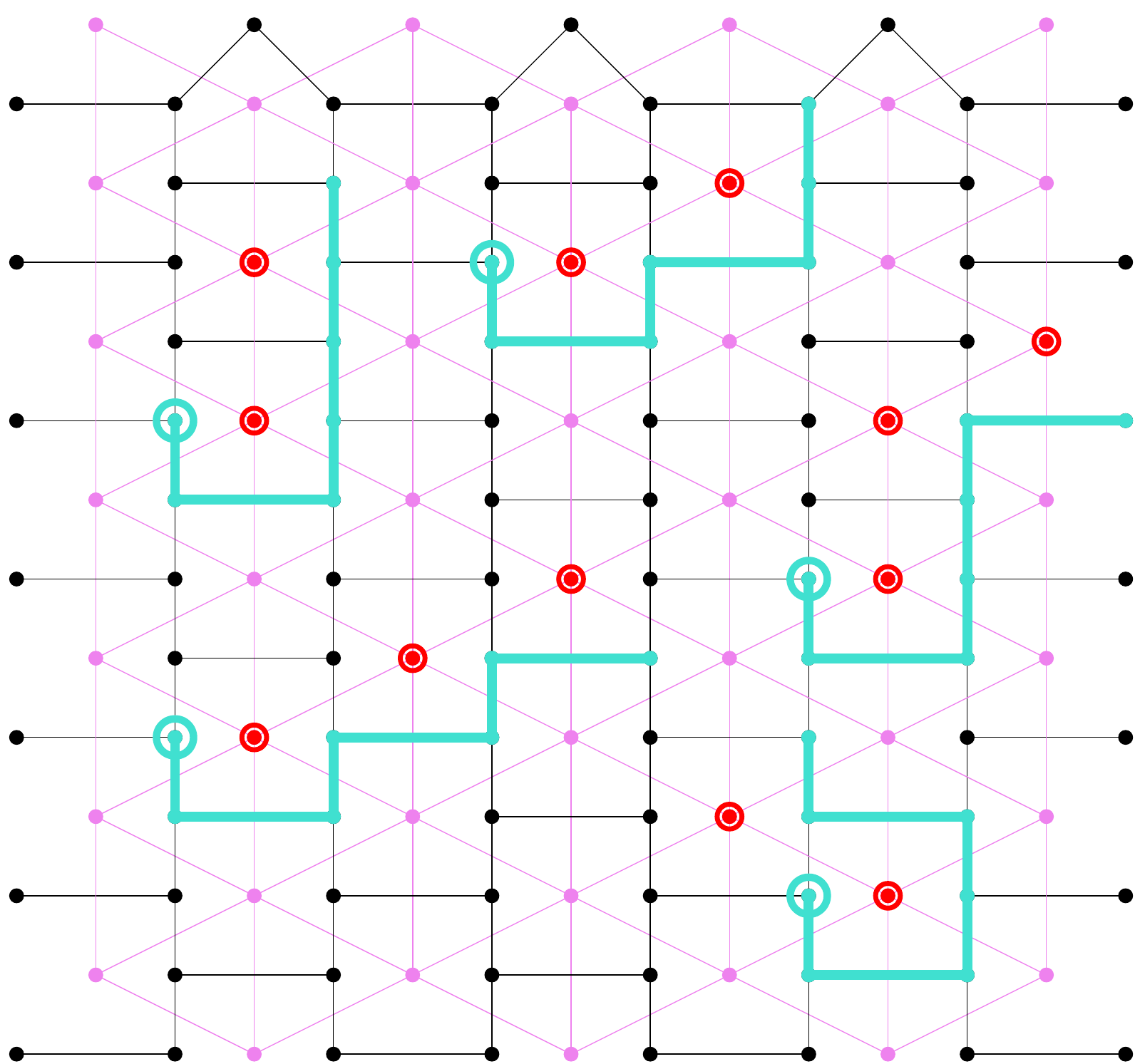}
		\caption{DualBrackets: Type 1, 2, 4, 3 and 5 clockwise from top right with $(x,y)$ marked with turquoise circle and interior dual-vertices in red}
	\end{figure}
	\end{description}
\end{defn}

Secure Floating Components are defined as they are for the triangular grid graph however with dualbrackets used instead of the original brackets. 

\begin{remark}
	All connected components $C$ of dual edges have a left most lowest dual-vertex. The dualbrackets are formed by being the 3 lowest edges whose dual are adjacent to this lowest dual-vertex and one of the 5 possible 3-edge-paths along the exterior dual edges of $C$.
\end{remark}

Before defining secure top and bottom components, we define top-securing vertices and bottom-securing vertices of $H_{n,m}$ to be the vertices of form ${(x,2m-2)}$ and ${(x,2)}$ respectively. We do this because edges above these vertices are never required to secure a top/bottom component, as if $\mathcal{V}$ chooses these edges it would create a dual arch.
	
\begin{defn}[Secure top component]
	A top component $C$ with unique top-most dual-vertex $t=(x, 2m-1)$ is extra-secured by blue path $P$ from a top-securing vertex $v$ to top-securing vertex $v'$ or secured by an unclaimed edge $g=(x+0.5,2m-2.5)$ and blue path $P$ from top-securing vertex $v$ to the lower endpoint of $g$ if:
	\begin{enumerate}[(i)]
		\item $C\backslash t$ is in the the interior of the arch $P\cup \{g\}$
		\item The dual $e*$ of each edge $e$ of $P$ is adjacent to a vertex of $C$ 
	\end{enumerate}
\end{defn}
	
\begin{defn}[Secure bottom component]
	A bottom component $C$ with unique bottom-most dual-vertex $b=(x, 1)$ is extra-secured by blue path $P$ from a bottom-securing vertex to bottom-securing vertex, or secured by an unclaimed edge $g=(x+0.5,2.5)$ and blue path $P$ from top-securing vertex $v$ to the upper endpoint of $g$ if:
	\begin{enumerate}[(i)]
		\item $C\backslash b$ is in the the interior of the arch $P\cup \{g\}$
		\item The dual $e*$ of each edge $e$ of $P$ is adjacent to a vertex of $C$ 
	\end{enumerate}
\end{defn}

\begin{figure}[h]
	\includegraphics[scale=0.6]{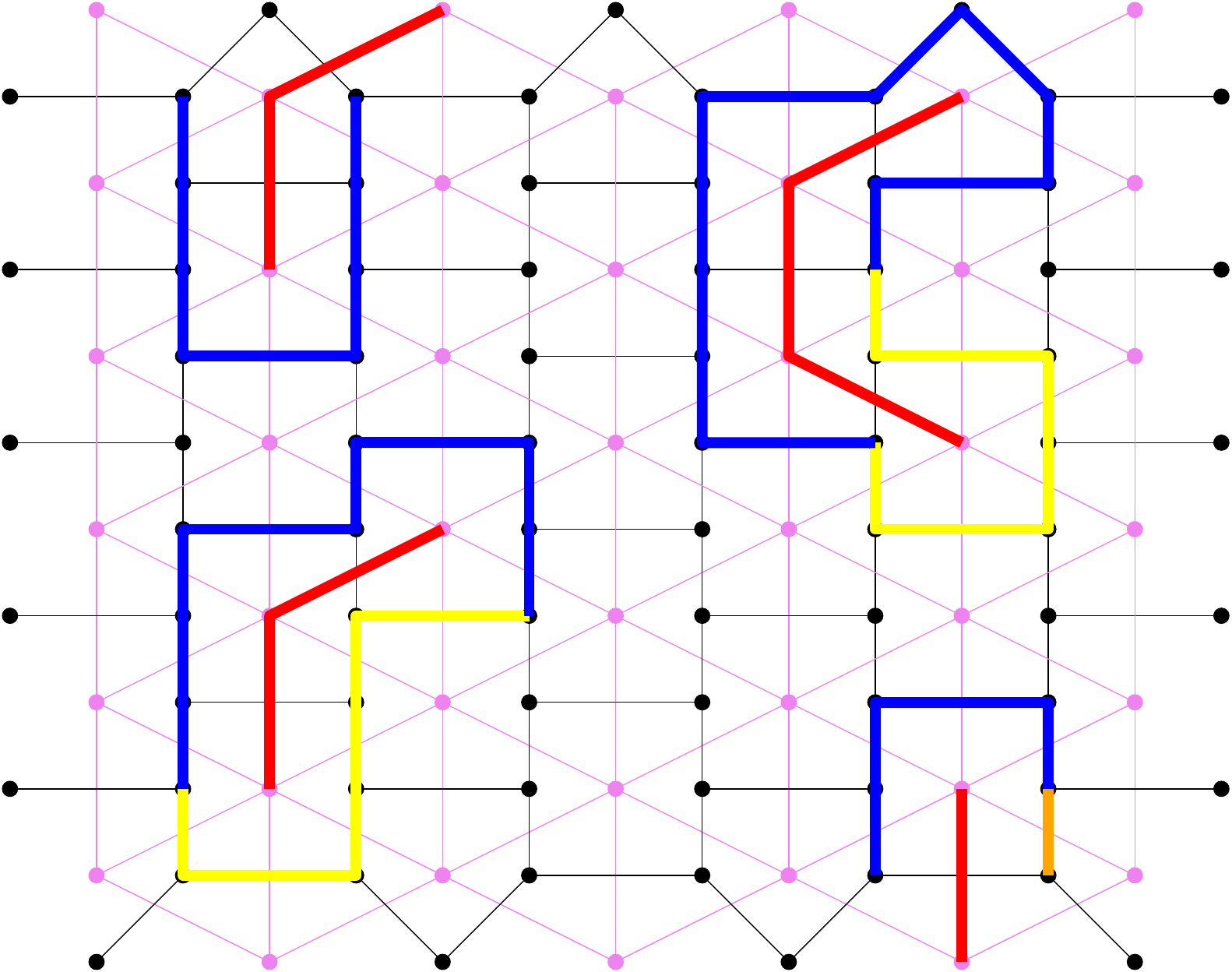}
	\caption{from top left clockwise: An extra-secure top component secured by blue path and, a secure floating component secured by blue path and Type 2 DualBracket, a secure bottom component secured by blue path and gate in orange and a secure floating component secured by blue path and Type 4 DualBracket}
\end{figure}

\begin{lem}
	Let $m\geq 1.5q+2$. If the grid is secure at the start of $\mathcal{V}$'s turn in the q-4response game played on $H_{n,m}$, then $\mathcal{V}$ cannot win in a single turn.
\end{lem}
\begin{proof}
	Let $m\geq (1.5)q$ and the grid be in a secure position before $\mathcal{V}$'s turn and he selects $q$ edges. 

	Suppose he claims $l \leq q$ edges ${e_1,e_2,\dots,e_l}$to create a bottom-top crossing path of red edges. Let the order of these edges be in their use from bottom to top in the path. 
	
	Define $y_i^-$ and $y_i^+$ to be the the vertical labelling of the vertices adjacent to edge $e_i$ in the order of appearance in the bottom top path. We will prove by induction that  $y_i^+ \leq 3i-1$.
	
	For the first edge it can either be adjacent to a bottom most dual-vertex, in which case $y_1^- = 1$ and $y_1^+ = 3$, or it's adjacent to a bottom most component and is dual to the gate helping secure the bottom component. By definition of bottom component this would mean $y_1^+ = 2$. Therefore the statement holds for $i=1$
	
	Suppose the statement is true for $i=k$ and $y_k^+ \leq 3k-1$. If $e_{k+1}$ is adjacent to $e_k$ then clearly $y_{k+1}^+ \leq y_k^+ +2 \leq 3(k+1)-1$. If not then $e_k$ is adjacent to an existing component. The component cannot be a bottom component or a top component else $e_{k+1}$ is not part of the path, so it is a floating component and $e_k$ and $e_{k+1}$ are dual edges of the dualbrackets. The value of $y_{k+1}^+ - y_k^+$ is bounded by the largest difference in vertical label between an interior dual-vertex of the component and a vertex not in $C$ adjacent to the dual of the bracket securing $C$. For dualbrackets it is easy to check that for Types 1, 2, 3, 4 and 5 these differences are 3, 2, 2, 1 and 0 respectively. 
	Hence, $y_{k+1}^+ - y_k^+ \leq 3$ and $y_{k+1}^+ \leq y_k^+ +3 \leq 3(k+1)-1$.
	
	Since the lower top dual-vertices have vertical label $2m-1\geq 3q+3$ and any top component must be reached through a gate to vertex with vertical label $2m-3 \leq 3q+1$. This contradicts that $\mathcal{V}$ creates a top bottom crossing path from a secure board in one turn.
\end{proof}

We define the secure game on the hexagonal grid graph as the same as on the triangular grid graph except $\mathcal{H}$ may claim 3 addition edges.

\begin{defn}{Secure game on $H_{n,m}$}
	The Secure game on $H_{n,m}$ is played similarly to that on $\Delta_{\infty,n}$. At any point an edge may be unclaimed, red, blue or a blue double edge (claimed by $\mathcal{H}$ twice).\\
	On each turn $\mathcal{V}$ claims ANY edge and makes it red, subject to the following conditions:
	\begin{enumerate}[(i)]
		\item $\mathcal{V}$ cannot claim an edge if doing so would create either a red dualcycle or a red dualarch (path from top dual-vertex to top dual-vertex or equivalently for bottom dual-vertices).
		\item $\mathcal{V}$ cannot claim an edge to connect a top component and bottom component
		\item If $C$ is a floating component and $P$ a blue path that 'secures' $C$, $\mathcal{V}$ cannot claim an edge of $P$ if doing so turns $C$ into a top or bottom component.
	\end{enumerate}
	Once $\mathcal{V}$ has claimed an edge $e$, $\mathcal{H}$ responds by claiming $4+b$ unclaimed edges, where $b$ is the number of blue edges replaced by $e$.\\
	At any stage the grid is secure if:
	\begin{enumerate}
		\item the grid is 'secure' with respect to the q-response game
		\item for distinct red components $C$ and $C'$ secured by blue paths $P$ and $P'$, any edge in $P \cap P'$ is a blue double edge.
	\end{enumerate}
	$\mathcal{H}$ wins if she can ensure the grid is secure at the end of each of their goes. $\mathcal{V}$ wins otherwise.
\end{defn}

\begin{lem}
	$\mathcal{H}$ has a winning strategy for the secure game on $H_{n,m}$ for $m\geq3$
\end{lem}
\begin{proof}
	Suppose the grid is secure before the beginning of $\mathcal{V}$'s turn and let $\mathcal{V}$ claim edge $e$ where $e*$ is adjacent to dual-vertices $v_1$ and $v_2$.\\
	
	\textbf{Case 1.} Let $v_1$ and $v_2$ not be part of any existing red components.\\
	If neither $v_1$ or $v_2$ are top or bottom vertices, then with dualbrackets Type 1,2 or 3 (depending on orientation of $v_1$ and $v_2$ as the interior dual-vertices) $\mathcal{H}$ can secure the new component with the 4 other edges whose duals are adjacent to $v_1$.
	If WLOG $v_1$ is a top (/bottom) dual-vertex, then either:
	\begin{itemize}
		\item $v_1$ and $v_2$ have different horizontal labels, in which case the dual edge can be extra secured by the 4 edges adjacent to $v_2$.
		\item or $v_1$ and $v_2$ have equal horizontal label, in which case the the dual edge can be secured by the 4 left most edges whose dual is adjacent to $v_2$ with the last one as the gate.
	\end{itemize}

	\textbf{Case 2.} Let $v_1$ be part of an existing red component $C$ and $v_2$ be not. Let $C$ be secured by blue path $P$ and bracket of non-red edges $B$.\\
	
	Case 2a. If $e$ lies in the interior then $C$ is still secure.\\
	
	Case 2b. If $e$ is a blue edge then by the conditions of the game $v_2$ cannot be a bottom or top vertex and $\mathcal{H}$ can claim 5 edges. She claims the non-red edges $f_1$, $f_2$, $f_3$, $f_4$ and $f_5$ whose dual are adjacent to $v_2$ to 're-secure' the component with the the new blue path $P-{e}+{f_1,f_2,f_3,f_4,f_5}$ and the same bracket $B$. \\
	
	Case 2c. If $e$ is the 'gate' of a top (or bottom) component $C$ then by the required position of the gate, the component can be made extra secure by $\mathcal{H}$ choosing the 4 unselected edges whose duals are adjacent to $v_2$.
	
	Case 2d. If $e$ is part of the dualbracket $B$ securing the existing red component then we consider the external boundary edges of the component.  $C \cup {e}$ has at most 10 external boundary edges. A new dualbracket can be found by finding the leftmost lowest dual-vertex of the {interior dual-vertices of $B$, $v_2$} and including the lower 3 external boundary edges of this vertex then selecting the next three external boundary edges right of this lowest dual-vertex. By the remark following the dualboundary definitions this is a dual bracket. $\mathcal{H}$ selects the other 4 external boundary edges to secure the component.

	\textbf{Case 3.} Let $v_1$ and $v_2$ be part of existing components $C_1$ and $C_2$ . Let $C_i$ be secured by blue path $P_i$ and either dualbracket of non-red edges $B_i$ or gate of non-red edge $g_i$ if $C_i$ is not secure. \\
	
	Clearly $C_1$ cannot equal $C_2$ else $e*$ creates a dual cycle in that connected component which breaks rule (i).\\
	$C_1$ and $C_2$ cannot both be top or bottom components else $e$ either creates a dual arch or connects a top and bottom component contradicting conditions (i) and (ii) respectively of the secure game.\\
	
	Case 3a. $C_1$ is a top or bottom component and $C_2$ is a floating component. By restriction (iii) $e$ must be in the dualbracket $B_2$. If $e$ is part of the blue path $P_1$ then player $\mathcal{H}$ has 5 moves and chooses the other 5 edges of $B_2$ to make $C_1 \cup C_2 \cup {e*}$ a secure top or bottom component. If $e$ is the gate $g_1$ then by trivial observation we can see that $e$ cannot be part of dualbracket $B_2$ for any dualbracket type.\\

	Case 3b. $C_1$ and $C_2$ are both floating components. If $e$ is part of both $P_1$ and $P_2$, then player $\mathcal{H}$ has 6 moves and uses them to claim all edges of dualbracket $B_1$  re-securing the component $C_1 \cup C_2$ with $P_1 \cup P_2 \cup B_1 \backslash e$ and dualbracket $B_2$. If $e$ is part of $P_1$ and $B_2$, then player $\mathcal{H}$ has 5 moves and uses them to claim the remaining 5 edges of bracket $B_2$ re-securing the component $C_1 \cup C_2$ with $P_1 \cup P_2 \cup B_2 \backslash e$ and $B_1$.\\
	Finally if $e$ is part of $B_1$ and $B_2$, then the set $B_1 \cup B_2 \backslash e$ consists of at most 10 non-red edges in one connected component. Consider the lowest interior dual-vertices of $B_1$ and $B_2$, $v'$. Without loss of generality let $v'$ belong to $C_1$. Then by the structure of the grid $e$ is not in the 3 lower edges of $B_1$. Following the remark after the definition of dualbrackets, we take the new bracket to be these 3 lower edges and the 3-edge-path from the right endpoint those lower edges. By construction this is a dual bracket. $\mathcal{H}$ claims the other 4 edges to re-secure the component. 
	
	Hence, at the end of player $\mathcal{H}$'s turn, the board is always secure. Hence $\mathcal{H}$ has a winning strategy for the secure game on $\Delta_{\infty,n}$ for $n\geq2$.
\end{proof}

\begin{thm}
	DualMaker has a winning strategy for the q-4response game on $H_{n,m}$ for $m\geq q+1$
\end{thm}
We apply exactly the same method to change the secure game strategy on $H_{n,m}$ into a q-4response game as that from the secure game on $\Delta_{m,n}$ into the q-response game.
The proof is hence analogous.

\section*{Acknowledgement}
This project was carried out with the support of The University of Warwick's Undergraduate Research Support Scheme (URSS). A particular thanks to Professor Oleg Pikhurko for his useful and stimulating discussions.
\bibliography{biblio}
		
\end{document}